\documentclass[10pt,reqno]{amsart}

\usepackage{a4wide}
\usepackage{dsfont}				
\usepackage[bbgreekl]{mathbbol}			
\usepackage{amssymb}
\usepackage{amsmath}				
\usepackage{mathrsfs}				
\usepackage[all]{xy}
\usepackage[backend=bibtex,  style=alphabetic, firstinits=true,doi=false,url=false,isbn=false,maxbibnames=99]{biblatex}

\renewbibmacro{in:}{} 

\renewbibmacro*{volume+number+eid}{%
  \printfield{volume}%
  \setunit*{\addnbspace}
  \printfield{number}%
  \setunit{\addcomma\space}%
  \printfield{eid}}

\DeclareFieldFormat[article,inbook,incollection,inproceedings,patent,thesis,unpublished]
{title}{{\em #1}\isdot}
\DeclareFieldFormat[article,inbook,incollection]{journaltitle}{#1}
\DeclareFieldFormat[article,inbook,incollection]{booktitle}{#1}
\DeclareFieldFormat[article]{volume}{\mkbibbold{#1}} 
\DeclareFieldFormat[article]{number}{\mkbibparens{#1}} 
\DeclareFieldFormat[article]{pages}{#1}

\DeclareSymbolFontAlphabet{\mathbb}{AMSb}	
\DeclareSymbolFontAlphabet{\mathbbl}{bbold}	

\newtheorem{proposition}{Proposition}[section]
  \newtheorem{theorem}[proposition]{Theorem}
  \newtheorem{corollary}[proposition]{Corollary}
  \newtheorem{lemma}[proposition]{Lemma}
\theoremstyle{definition}
  \newtheorem{definition}[proposition]{Definition}
  \newtheorem{remark}[proposition]{Remark}

\newcommand{\cst}{\ifmmode\mathrm{C}^*\else{$\mathrm{C}^*$}\fi}
\newcommand{\st}{\;\vline\;}
\newcommand{\tens}{\otimes}
\newcommand{\vtens}{\,\bar{\tens}\,}
\newcommand{\atens}{\tens_{\text{\rm{\tiny{alg}}}}}
\newcommand{\id}{\mathrm{id}}
\newcommand{\comp}{\circ}
\newcommand{\I}{\mathds{1}}
\newcommand{\bh}{\boldsymbol{h}}
\newcommand{\cc}{\mathrm{c}}
\newcommand{\qqquad}{\quad\qquad}
\newcommand{\tp}{\;\!\xymatrix{*+<.1ex>[o][F-]{\raisebox{-1.4ex}{$\scriptstyle\top$}}}\;\!}
\newcommand{\ww}{\mathrm{W}}
\newcommand{\Limits}{\limits}			
\newcommand{\hh}[1]{\widehat{#1}}
\newcommand{\is}[2]{\left\langle#1\,\vline\,#2\right\rangle}

\newcommand{\CC}{\mathbb{C}}
\newcommand{\NN}{\mathbb{N}}
\newcommand{\GG}{\mathbb{G}}
\newcommand{\HH}{\mathbb{H}}
\newcommand{\XX}{\mathbb{X}}
\newcommand{\GGamma}{\mathbbl{\Gamma}}
\newcommand{\LLambda}{\mathbbl{\Lambda}}

\newcommand{\cZ}{\mathscr{Z}}

\newcommand{\sA}{\mathsf{A}}
\newcommand{\sM}{\mathsf{M}}
\newcommand{\sN}{\mathsf{N}}
\newcommand{\sH}{\mathsf{H}}
\newcommand{\sK}{\mathsf{K}}

\newcommand{\Ralpha}{\sim_\alpha}
\newcommand\Ind{\mathcal{I}}
\newcommand\Jnd{\mathcal{J}}
\newcommand{\bp}{\boldsymbol{p}}
\newcommand{\bpi}{\boldsymbol{\pi}}
\newcommand{\RLambda}{\sim_{\LLambda}}

\DeclareMathOperator{\B}{B}
\DeclareMathOperator{\Irr}{Irr}
\DeclareMathOperator{\supp}{supp}
\DeclareMathOperator{\Pol}{Pol}
\DeclareMathOperator{\Linf}{\mathnormal{L}^\infty\;\!\!}
\DeclareMathOperator{\linf}{\ell^\infty\;\!\!}
\DeclareMathOperator{\Ltwo}{\mathnormal{L}^2\;\!\!}
\DeclareMathOperator{\c0}{c_0}
\DeclareMathOperator{\Tr}{Tr}
\DeclareMathOperator{\z}{\mathnormal{z}}
\DeclareMathOperator{\mult}{mult}

\numberwithin{equation}{section}

\begin{document}

\author{Kenny De Commer}
\address{Vakgroep Wiskunde, Vrije Universiteit Brussel, Belgium}
\email{kenny.de.commer@vub.ac.be}

\author{Pawe{\l} Kasprzak}
\address{Department of Mathematical Methods in Physics, Faculty of Physics, University of Warsaw, Poland}
\email{pawel.kasprzak@fuw.edu.pl}

\author{Adam Skalski}
\address{Institute of Mathematics of the Polish Academy of Sciences, Warsaw, Poland}
\email{a.skalski@impan.pl}

\author{Piotr M.~So{\l}tan}
\address{Department of Mathematical Methods in Physics, Faculty of Physics, University of Warsaw, Poland}
\email{piotr.soltan@fuw.edu.pl}

\title[Quantum Clifford theory]{Quantum actions on discrete quantum spaces and a generalization of Clifford's theory of representations}



\begin{abstract}
To any action of a compact quantum group on a von Neumann algebra which is a direct sum of factors we associate an equivalence relation corresponding to the partition of a space into orbits of the action. We show that in case all factors are finite-dimensional (i.e.~when the action is on a discrete quantum space) the relation has finite orbits. We then apply this to generalize the classical theory of Clifford, concerning the restrictions of representations to normal subgroups, to the framework of quantum subgroups of discrete quantum groups, itself extending the context of closed normal quantum subgroups of compact quantum groups. Finally, a link is made between our equivalence relation in question and another equivalence relation defined by R.~Vergnioux.
\end{abstract}

\maketitle

\section{Introduction}

The story and motivation behind this paper mirror to an extent those behind its classical, almost eighty years old predecessor: Clifford theory, as developed in the article \cite{clifford}, was clearly inspired by Frobenius's induced representations. Let us recall that \cite{clifford} concerns questions around the decomposition of the restriction of a given irreducible representation of a group $G$ to a normal subgroup $H\subset G$ (for a modern treatment of the results of that paper for finite groups we refer to \cite{Ceccherini}). Clifford considers general groups, but only finite-dimensional representations. We will mostly focus on the compact case, where finite dimensionality is automatic. The key notion appearing in this context is that of orbits of the adjoint action of $G$ on $\Irr{H}$, the set of (equivalence classes of) irreducible representations of $H$.

Since the arrival of compact quantum groups on the scene in the 1980s and the development of a satisfactory theory of locally compact quantum groups at the turn of the century, there has been a lot of interest in the study of various aspects of representation theory in this, quantum, context. In particular the articles \cite{KustermansInduced} and \cite{Vaes-induction} developed in full a quantum counterpart of the induction theory of Mackey -- itself generalizing the aforementioned work of Frobenius for finite groups. The induction mechanisms proposed by Kustermans and Vaes are necessarily somewhat complicated, and in \cite{VergniouxVoigt} a simpler picture was established for discrete quantum groups. On the other hand, in the recent article \cite{KalKS} a notion of an open quantum subgroup was proposed (encompassing in particular all quantum subgroups of discrete quantum groups), allowing for a framework in which one can use Rieffel's approach to induction (\cite{KKSinduced}). All these developments made natural the need to understand better the relation between open and discrete induction, as studied respectively in \cite{KKSinduced} and \cite{VergniouxVoigt}, and brought us to developing a quantum counterpart of Clifford theory.

Classical Clifford theory deals with the relationship between representations of a group $G$ and their restrictions to a normal subgroup $H$ of $G$, the main tool being the action of $G$ on irreducible representations of $H$ by composition with a $G$-inner automorphism. In our approach we first look at a pair $(\GG,\HH)$ consisting of a compact quantum group $\GG$ and its closed normal quantum subgroup $\HH$. This pair gives rise to a discrete quantum group $\GGamma=\hh{\GG}$ with a quantum subgroup $\LLambda=\hh{\GG/\HH}$. The quantum subgroup $\LLambda\subset\GGamma$ is normal (\cite[Page 46]{extensions}) and we can recover $(\GG,\HH)$ as $\GG=\hh{\GGamma}$ and $\HH=\hh{\GGamma/\LLambda}$. Restriction of representations from $\GG$ to $\HH$ translates to restricting representations of $\linf(\GGamma)$ to the subalgebra $\linf(\GGamma/\LLambda)$. Our generalization will consist of dropping the assumption of normality of $\LLambda$. In particular this introduces a difference between left and right quotients $\LLambda\backslash\GGamma$ and $\GGamma/\LLambda$ and it turns out that $\LLambda\backslash\GGamma$ is more suited to our conventions. This difference, however, is not of importance, and in any case we have $R^\GGamma\bigl(\linf(\GGamma/\LLambda)\bigr)=\linf(\LLambda\backslash\GGamma)$, where $R^\GGamma$ is the unitary antipode of $\GGamma$. In this context we will introduce an action of the compact quantum group $\GG$ on $\LLambda\backslash\GGamma$ corresponding to the action known from Clifford theory, and introduce the notion of \emph{orbit} for this action. Then we will be in a position to prove a quantum analog of Clifford's theorem on restriction of irreducible representations.

Let us briefly describe the contents of the paper. In Section \ref{intro} we introduce the necessary language and notation as well as list some standard results from the theory of compact and discrete quantum groups. Section \ref{Kenny} contains a proof of a fundamental result saying that if a compact quantum group acts ergodically on a von Neumann algebra $\sN$ with a finite-dimensional direct summand then $\sN$ itself must be finite-dimensional (Theorem \ref{mainthm}). In Section \ref{Adam} we introduce and study the quantum group analog of the relation of being in the same orbit with respect to a compact quantum group action on a direct sum of factors. In Section \ref{QCliffordTheory} we give two applications of the theory. First we generalize the main result of Clifford theory concerning restrictions of representations to a normal subgroup. Then, in Section \ref{VVeq} we exhibit the connection of the equivalence arising in quantum Clifford theory with an equivalence relation of Vergnioux (\cite{orientation}).

\section{Notation and preliminaries}\label{intro}

For a von Neumann algebra $\sM$ the Banach space of normal functionals on $\sM$ will be denoted by $\sM_*$, and for $\omega\in\sM_*$ we define $\overline{\omega}\in\sM_*$ by the standard formula $\overline{\omega}(x)=\overline{\omega(x^*)}$ for all $x\in\sM$.
The center of $\sM$ will be denoted $\cZ(\sM)$. The tensor product of von Neumann algebras will be denoted $\vtens$ while the symbol $\tens$ will be reserved for minimal tensor products of \cst-algebras and for tensor products of maps in various settings. On one occasion we will be considering a tensor product of a von Neumann algebra with a $\sigma$-weakly closed operator space. The symbol $\vtens$ will in this case denote the $\sigma$-weak completion of the algebraic tensor product in its natural spatial implementation.

All scalar products will be linear on the right. We will often use Sweedler notation and leg numbering notation familiar from Hopf algebra theory and quantum group literature. For a Hilbert space $\sH$ and vectors $\xi,\eta\in\sH$ the symbol $\omega_{\xi,\eta}$ will denote the continuous functional on $\B(\sH)$ given by the formula $T\mapsto\is{\xi}{T\eta}$.

Given a family of \cst-algebras $\{\sA_i\}_{i\in\Ind}$ the symbol $\bigoplus\Limits_{i\in\Ind}\sA_i$ will denote the $\c0$-direct sum of the family $\{\sA_i\}_{i\in\Ind}$, while $\prod\Limits_{i\in\Ind}\sA_i$ will be the $\ell^\infty$-direct sum of the same family. We remark that if $(\sM_i)_{i\in\Ind}$ is a family of matrix algebras then any (not necessarily unital) $\sigma$-weakly closed $*$-subalgebra of $\prod\Limits_{i\in\Ind}\sM_i$ is again isomorphic to a product of matrix algebras.

Throughout the paper $\GG$ will denote a compact quantum group. We will study it mainly via the associated von Neumann algebra $\Linf(\GG)$, representing the ``algebra of bounded measurable functions on $\GG$'', equipped with coproduct $\Delta_\GG:\Linf(\GG)\to\Linf(\GG)\vtens\Linf(\GG)$. The Haar state on $\GG$ will be denoted by $\bh_\GG$ and $\Linf(\GG)$ is assumed to act on the GNS space of $\bh_\GG$ denoted $\Ltwo(\GG)$. A \emph{(finite-dimensional) unitary representation} $\pi$ of $\GG$ on a (finite-dimensional) Hilbert space $\sH$ is a unitary matrix $U = U^{\pi}\in\B(\sH)\tens\Linf(\GG)$ such that $(\id\tens\Delta_\GG)U=U_{12}U_{13}$. Its \emph{coefficients} are elements of the form $(\omega_{\xi,\eta}\tens\id)U$, with $\xi,\eta\in\sH$. There are natural notions of unitary equivalence and irreducibility of representations of compact quantum groups (\cite{cqg}), and any irreducible representation is finite-dimensional. We will denote by $\Irr{\GG}$ the set of all equivalence classes of irreducible representations of $\GG$ and tacitly assume that for every class $\iota\in\Irr{\GG}$ a representative has been chosen and fixed. We will often denote this representative by the same symbol as the class itself. Furthermore, for $\iota\in\Irr{\GG}$ the symbol $n_\iota$ denotes the (classical) dimension of the corresponding representation. The keystone of Woronowicz's quantum Peter-Weyl theory is the fact that coefficients of irreducible representations of $\GG$ span a dense unital $*$-subalgebra of $\Linf(\GG)$, denoted $\Pol(\GG)$, and $\bigl(\Pol(\GG),\bigl.\Delta_\GG\bigr|_{\Pol(\GG)}\bigr)$ is a Hopf $*$-algebra, whose antipode we will denote by $S$. We say that $\GG$ is finite if $\Linf(\GG)$ (equivalently $\Pol(\GG)$) is finite-dimensional, and that $\GG$ is of Kac type if $\bh_\GG$ is a tracial state.

The algebra $\Linf(\GG)$ is also equipped with a map $R^\GG$, called the \emph{unitary antipode}, leaving $\Pol(\GG)$ invariant, which may be used to define the \emph{contragredient representation} $\pi^\cc$ of a given representation $\pi$, see \cite{pseudogr,cqg}. There is also a natural notion of a tensor product of two representations, which we denote $\pi_1\tp\pi_2$. The operations of taking contragredient representations and tensor product pass to equivalence classes.

To every compact quantum group $\GG$ we associate the dual \emph{discrete} quantum group $\GGamma = \hh{\GG}$, which is usually studied either via its algebra $\c0(\GGamma)\cong\bigoplus\Limits_{\iota\in\Irr{\GG}}M_{n_\iota}(\CC)$ of ``functions vanishing at infinity'' or its algebra $\linf(\GGamma)\cong\prod\Limits_{\iota\in\Irr{\GG}}M_{n_\iota}(\CC)$ of ``bounded functions''. Both of these algebras are then equipped with a natural coproduct $\Delta_{\GGamma}$, allowing to take the tensor product of (normal) algebra representations, and a unitary antipode $R^{\GGamma}$. There is then a natural one-to-one correspondence between (irreducible) representations of $\GG$ and (irreducible) normal representations of the von Neumann algebra $\linf(\GGamma)$, respecting tensor products and contragredients of representations. More precisely, each representation $U$ of $\GG$ is of the form $U=U^{\pi}=(\pi\tens\id)(\ww^\GG)$ for a normal representation $\pi$ of $\linf(\GGamma)$, where $\ww^\GG\in\Linf(\GGamma)\vtens\Linf(\GG)$ is the \emph{right regular representation} (\cite[Section 3]{q-lorentz}). Then tensor product of $U^\pi$ and $U^\sigma$ is $U^{(\pi\tens\sigma)\comp\Delta_\GGamma}$ and $(U^\pi)^\cc=U^{\top\circ \pi\comp{R^\GGamma}}$ up to equivalence, with $\top$ the transpose operation. This justifies the notation introduced above, and in what follows we will identify $U$ with the corresponding representation $\pi$.

By an \emph{action} of a compact quantum group $\GG$ on a von Neumann algebra $\sM$ we mean an injective, normal unital $*$-homomorphism $\alpha:\sM\to\sM\vtens\Linf(\GG)$ satisfying the condition
\begin{equation}\label{actionEquation}
(\alpha\tens\id_{\Linf(\GG)})\comp\alpha=(\id_\sM\tens\Delta_\GG)\comp\alpha.
\end{equation}

We shall often write $\Linf(\XX)$ for $\sM$, thus defining a ``quantum (measure) space'' $\XX$. We will also write $\Pol(\XX)$ for the \emph{algebraic core} of $\Linf(\XX)$ (see \cite{KennyLectures,podles}). It is a dense unital $*$-subalgebra of $\Linf(\XX)$ (called also the \emph{Podle\'s subalgebra}) such that $\alpha$ restricts to a coaction $\bigl.\alpha\bigr|_{\Pol(\XX)}:\Pol(\XX)\to\Pol(\XX)\atens\Pol(\GG)$ of the Hopf algebra $\Pol(\GG)$. We use a variation of Sweedler notation: for $x\in\Pol(\XX)$ we write $\alpha(x)=x_{(0)}\tens{x_{(1)}}$.

The action $\alpha$ by a compact quantum group $\GG$ as above is said to be \emph{ergodic} if the fixed point subalgebra $\sM^\alpha=\bigl\{x\in\sM\st\alpha(x)=x\tens\I\bigr\}$ equals $\CC\I$. Then $\Linf(\XX)$ admits a unique $\alpha$-invariant state, which we will denote $\bh_\XX$. It is determined by the following condition
\[
\bh_\XX(x)\I=(\id\tens\bh_\GG)\bigl(\alpha(x)\bigr),\qqquad{x}\in\Linf(\XX).
\]
The GNS Hilbert space of $\bh_\XX$ will be denoted by $\Ltwo(\XX)$. In what follows we shall view $\Linf(\XX)$ as a von Neumann subalgebra of $\B(\Ltwo(\XX))$.

Given two compact quantum groups $\GG$ and $\HH$ we say that $\HH$ is a (closed quantum) subgroup of $\HH$ if there exists a surjective Hopf $*$-algebra map $\Pol(\GG)\to\Pol(\HH)$. This is equivalent to the existence of an injective normal embedding, respecting the coproducts, of the von Neumann algebra $\linf(\hh{\HH})$ into $\linf(\hh{\GG})$. Note that in that case we can naturally talk about restricting representations of $\GG$ to $\HH$, which is the main point of interest for this paper. The subgroup $\HH$ is said to be normal if in addition $\ww^\GG\bigl(\linf(\hh{\HH})\tens\I\bigr){\ww^\GG}^*\subset\linf(\hh{\HH})\vtens\Linf(\GG)$ (\cite{WangFree,Wang2014,centers}), in which case one can define a quotient compact quantum group $\GG/\HH$. Similarly given two discrete quantum groups $\GGamma$ and $\LLambda$ we say that $\LLambda$ is a (closed quantum) subgroup of $\GGamma$ if we have an injective Hopf $*$-algebra map $\Pol(\hh{\LLambda})\to\Pol(\hh{\GGamma})$. This is equivalent to the existence of a surjective $*$-homomorphism from $\c0(\GGamma)$ onto $\c0(\LLambda)$, once again intertwining respective coproducts. Most of the above concepts have natural generalizations to arbitrary locally compact quantum groups, which we will briefly use in the beginning of Section \ref{QCliffordTheory}.

For more information on actions of compact quantum groups and related topics we refer the reader to the lecture notes \cite{KennyLectures}. The topics of quantum subgroups are thoroughly covered in \cite{DKSS} and information on normal quantum subgroups, inner automorphisms etc.~can be found in \cite{centers} and references therein. We will at some point use the theory of locally compact quantum groups in the sense of Kustermans and Vaes (\cite{KV}).

\section{Ergodic actions of a compact quantum group on quantum spaces with a discrete component}\label{Kenny}

In this section we show that if a compact quantum group acts ergodically on a ``quantum space $\XX$ with a discrete component'', that is on a von Neumann algebra $\Linf(\XX)$ of the form ${M_n(\CC)}\oplus\sN$ for some $n\in\NN$, then $\XX$ must be finite (i.e.~$\Linf(\XX)$ is finite-dimensional).

Let $\Linf(\XX)$ be a von Neumann algebra and $\alpha:\Linf(\XX)\to\Linf(\XX)\vtens\Linf(\GG)$ an ergodic action. Denote by $\Lambda_\XX$ and $\Lambda_\GG$ the GNS maps corresponding to the invariant state $\bh_\XX$ on $\Linf(\XX)$ and the Haar measure $\bh_\GG$ and consider the isometry
\[
G:\Ltwo(\XX)\tens\Ltwo(\XX)\longrightarrow\Ltwo(\XX)\tens\Ltwo(\GG),
\]
defined by
\[
\Lambda_{\XX}(x)\tens\Lambda_{\XX}(y)\longmapsto(\Lambda_{\XX}\tens\Lambda_{\GG})\bigl(\alpha(y)(x\tens\I)\bigr),\qqquad{x,y}\in\Pol(\XX).
\]
Since
\[
(\id\tens\omega_{\Lambda_{\XX}(z),\Lambda_{\XX}(y)})(G)=(\id\tens\bh_{\GG})\bigl((\I\tens{z^*})\alpha(y)\bigr),\qqquad y,z\in\Pol(\XX),
\]
it follows that
\[
G\in\Linf(\XX)\vtens\B(\Ltwo(\XX),\Ltwo(\GG)),
\]
so in particular we can consider for $\omega\in\Linf(\XX)_*$ the element
\[
(\omega\tens\id)(G)\in\B(\Ltwo(\XX),\Ltwo(\GG)).
\]
An easy computation shows that
\begin{equation}
(\omega\tens\id)(G)\Lambda_\XX(x)=\Lambda_\XX\bigl((\omega\tens\id)(\alpha(x)\bigr),\qqquad{x}\in\Linf(\XX).
\label{adjoint} \end{equation}

For each $\omega\in\Linf(\XX)_*$ consider
\[
L_{\omega}=(\omega\tens\id)(G^*):\Ltwo(\GG)\longrightarrow\Ltwo(\XX).
\]
Note that using \eqref{adjoint} one can easily check that the map $\omega \mapsto L_\omega$ is injective.

Let $\c0(\hh{\XX})$ denote the norm-closure of the set of the operators $L_\omega$:
\[
\c0(\hh{\XX})=\overline{\bigl\{(\omega\tens\id)(G^*)\st\omega\in\Linf(\XX)_*\bigr\}}^{\scriptscriptstyle{\|\cdot\|}}.
\]
The notation is justified by the fact that $\hh{\XX}$ is in a sense a discrete object, as follows from Lemma \ref{lem:discreteness} below.

Applying the construction above to $\XX=\GG$ we get a copy of $\c0(\hh{\GG})$ (more precisely the \cst-algebra associated to the left regular representation of $\GG$ on $\Ltwo(\GG)$). In particular, for $\chi\in\Linf(\GG)_*$ we shall write $L_\chi\in\c0(\hh{\GG})$.

\begin{lemma}
For $a\in\Pol(\GG)$, $x,y\in\Pol(\XX)$, $\omega\in\Linf(\XX)_*$ and $\chi\in\Linf(\GG)_*$ the following equalities hold:
\begin{subequations}
\begin{align}
L_{\bh_{\XX}(\,\cdot\,x)}\Lambda_{\GG}(a)&=\bh_{\GG}\bigl(S(x_{(1)})a\bigr)\Lambda_{\XX}(x_{(0)}),\label{eq:prop1}\\
L_{\omega}L_{\chi}&=L_{\omega\cdot\chi},\label{eq:prop2}\\
L_{\bh_{\XX}(\,\cdot\,y)}^*L_{\bh_{\XX}(\,\cdot\,x)}&=\bh_{\XX}(y^*x_{(0)})L_{\bh_{\GG}(\,\cdot\,x_{(1)})}
=\bh_{\XX}(y_{(0)}^*x)L_{\bh_{\GG}(\,\cdot\,y_{(1)})}^*,\label{EqInProd}
\end{align}
where $\omega\cdot\chi=(\omega\tens\chi)\comp\alpha$.
\end{subequations}
\end{lemma}

\begin{proof}
In order to get \eqref{eq:prop1} we compute
\begin{align*}
\is{\Lambda_{\XX}(y)}{L_{\bh_{\XX}(\,\cdot\,x)}\Lambda_{\GG}(a)}&=\is{\Lambda_{\XX}(\I)\tens\Lambda_{\XX}(y)}{G^*\Lambda_{\XX}(x)\tens\Lambda_{\GG}(a)}\\
&=\is{\Lambda_{\XX}(y_{(0)})\tens\Lambda_{\GG}(y_{(1)})}{\Lambda_{\XX}(x)\tens\Lambda_{\GG}(a)}\\
&=\bh_{\XX}(y_{(0)}^*x)\bh_{\GG}(y_{(1)}^*a)\\
&=\bh_{\XX}(y_{(0)}^*x_{(0)})\bh_{\GG}\bigl(y_{(1)}^*x_{(1)}S(x_{(2)})a\bigr)\\
&=\bh_{\XX}(y ^*x_{(0)})\bh_{\GG}\bigl(S(x_{(1)})a\bigr)\\
&=\is{\Lambda_{\XX}(y)}{\bh_{\GG}\bigl(S(x_{(1)})a\bigr)\Lambda_{\XX}(x_{(0)})},
\end{align*}
where in the fifth equality we used $\alpha$-invariance of $\bh_{\XX}$.

In order to prove \eqref{eq:prop2} we compute
\begin{align*}
L_{\chi}^*L_{\omega}^*\Lambda_{\XX}(x)&=\overline{\omega}(x_{(0)})L_{\chi}^*\Lambda_{\GG}(x_{(1)})\\
&=\overline{\omega}(x_{(0)})\overline{\chi}(x_{(1)})\Lambda_{\GG}(x_{(2)})\\
&=L_{\omega\cdot\chi}^*\Lambda_{\XX}(x).
\end{align*}

Finally, in order to prove the first equality of \eqref{EqInProd} we compute
\begin{align*}
L_{\bh_{\XX}(\,\cdot\,y)}^*L_{\bh_{\XX}(\,\cdot\,x)}\Lambda_{\GG}(a)
&=L_{\bh_{\XX}(\,\cdot\,y)}^*\bh_{\GG}\bigl(S(x_{(1)})a\bigr)\Lambda_{\XX}(x_{(0)})\\
&=\bh_{\XX}(y^*x_{(0)})\bh_{\GG}\bigl(S(x_{(2)})a\bigr)\Lambda_{\GG}(x_{(1)})\\
&=\bh_{\XX}(y^*x_{(0)})L_{\bh_{\GG}(\,\cdot\,x_{(1)})}\Lambda_{\GG}(a).
\end{align*}
The second equality of \eqref{EqInProd} follows from the first one by taking adjoints of both sides.
\end{proof}

\begin{corollary}
For all $a,b\in\c0(\hh{\XX})$ and $c\in\c0(\hh{\GG})$ we have $a^*b\in\c0(\hh{\GG})$ and $ac\in\c0(\hh{\XX})$. In particular $\c0(\hh{\XX})$ forms a Hilbert \cst-module over $\c0(\hh{\GG})$.
\end{corollary}

In the following, we will denote by $\Pol(\GG)^{\iota}$ the finite-dimensional space of matrix coefficients of $U^{\iota}$ and write
\[
\Pol(\XX)^{\iota}=\bigl\{(\id\tens\bh_{\GG}(S^{-1}(a)\,\cdot\,)\bigr)\alpha(x)\st{a}\in\Pol(\GG)^{\iota},\:x\in\Linf(\XX)\bigr\}
\]
for the corresponding \emph{spectral subspace}, which is finite-dimensional by \cite[Theorem 17]{boca}.

Recall that we have
\[
\c0(\hh{\GG})\cong\bigoplus_{\iota\in\Irr{\GG}}M_{n_\iota}(\CC).
\]

\begin{lemma} \label{lem:discreteness}
There is an isomorphism of $\c0(\hh{\GG})$-modules
\begin{equation}
\c0(\widehat{\XX})=\bigoplus_{\iota\in\Irr{\GG}}M_{m_\iota,n_\iota}(\CC) \label{c_0Xformula}
\end{equation}
for certain finite $m_\iota\in\NN$ (with possibly $m_\iota=0$) and the obvious Hilbert \cst-module structure.
\end{lemma}

\begin{proof}
It is easy to see that any Hilbert C$^*$-module over $\bigoplus\Limits_{\iota\in\Irr{\GG}}M_{n_\iota}(\CC)$ must be of the form $\bigoplus\Limits_{\iota\in\Irr{\GG}}B(\sH^{\iota},\sK^{\iota})$ for certain Hilbert spaces $\sK^{\iota}$, with $\sH^{\iota}$ the Hilbert space underlying the representation $U^\iota$. Our aim is to show that the $\sK^{\iota}$ are finite-dimensional.

However, denoting by $\Ltwo(\GG)^{\iota}$ and $\Ltwo(\XX)^{\iota}$ the finite-dimensional Hilbert spaces obtained as the images of $\Pol(\GG)^{\iota}$ and $\Pol(\XX)^{\iota}$ under the GNS-construction, it follows that each $x\in\c0(\widehat{\XX})$ sends $\Ltwo(\GG)^{\iota}$ into $\Ltwo(\XX)^{\iota}$. Hence $B(\sH^{\iota},\sK^{\iota})$ and so also $\sK^{\iota}$ is finite-dimensional.
\end{proof}

Note that it follows from the above lemma that for each $\iota \in \Irr{\GG}$
\begin{equation}\label{EqSameComp}
M_{m_{\iota},n_{\iota}}(\CC)\cong\bigl\{L_{\bh_{\XX}(\,\cdot x)}\st{x}\in\Pol(\XX)^{\iota}\bigr\}.
\end{equation}

\begin{theorem}\label{mainthm}
Let $\alpha$ be an ergodic action of a compact quantum group $\GG$ on $\XX$. If $\Linf(\XX)=M_n(\CC)\oplus\sN$ then $\dim{\Linf(\XX)}<\infty$.
\end{theorem}

\begin{proof}
Assume that $\Linf(\XX)$ is not of finite dimension, but
\[
\Linf(\XX)=M_n(\CC)\oplus\sN.
\]
We will arrive at a contradiction.

Denote by $\eta$ the canonical unital normal $*$-homomorphism $\eta:\Linf(\XX)\to{M_n(\CC)}$ obtained by projecting onto the first component. Consider the unital normal $*$-homomorphism $\alpha_{\eta}:\Linf(\XX)\to{M_n(\CC)}\tens\Linf(\GG)$ defined by
\[
x\longmapsto{x_\eta}=(\eta\tens\id)\alpha(x).
\]
Then for $x\in\Linf(\XX)$ define $\{x\}_{i,j=1,\dotsc,n}$ in $\Linf(\GG)$ by the equality
\[
x_\eta=\sum_{i,j=1}^ne_{ij}\tens x_{ij},
\]
where $\{e_{ij}\}_{i,j=1,\dotsc,n}$ are the matrix units of $M_n(\CC)$. Consider now the convolution operator $L_{(\id\tens\bh_{\GG})(\,\cdot\,x_{\eta})}:\CC^n\tens\Ltwo(\GG)\to\CC^n\tens\Ltwo(\GG)$ defined by
\[
L_{(\id\tens\bh_{\GG})(\,\cdot\,x_{\eta})}
(\xi\tens\zeta)=\sum_{i,j=1}^n e_{ij}\xi\tens L_{\bh_{\GG}(\,\cdot\,x_{ij})}\zeta,\qqquad\xi\in\CC^n,\;\zeta\in\Ltwo(\GG).
\]

We claim that
\begin{equation}\label{EqNorm}
\bigl\|L_{(\id\tens \bh_{\GG})(\,\cdot\,x_{\eta})}\bigr\|=\bigl\|L_{\bh_{\XX}(\,\cdot\, x)}\bigr\|.
\end{equation}
Indeed, we compute (using \eqref{EqInProd})
\[
\begin{split}
L_{(\id\tens\bh_{\GG})(\,\cdot\,y_{\eta})}^*L_{(\id\tens\bh_{\GG})(\,\cdot\,x_{\eta})}
&=\sum_{i,j,k,l=1}^n e_{ji}e_{kl}\tens{L_{\bh_{\GG}(\,\cdot\,y_{ij})}^*}L_{\bh_{\GG}(\,\cdot\,x_{kl})}\\
&=\sum_{j,l,i=1}^n e_{jl}\tens\bh_{\GG}(y_{ij}^*x_{il}{}_{(1)})L_{\bh_{\GG}(\,\cdot\,x_{il}{}_{(2)})}.
\end{split}
\]
However, we have
\[
\begin{split}
(\id\tens\bh_{\GG}\tens\id)\biggl(\sum_{j,l,i=1}^n e_{jl}\tens{y_{ij}^*}x_{il}{}_{(1)}\tens{x_{il}}{}_{(2)}\biggr)
&=\sum_{i,j,k,l=1}^n e_{ji}e_{kl}\tens\bh_{\GG}(y_{ij}^*x_{kl}{}_{(1)})\tens{x_{kl}{}_{(2)}}\\
&=(\id\tens\bh_{\GG}\tens\id)\bigl((\alpha_{\eta}(y)^*\tens\I)(\id\tens\Delta_\GG)\alpha_{\eta}(x)\bigr)\\
&=(\eta\tens\bh_{\GG}\tens\id)(\alpha\tens\id)\bigl((y^*\tens\I)\alpha(x)\bigr)\\
&=\I\tens(\bh_{\XX}\tens\id)\bigl((y^*\tens\I)\alpha(x)\bigr)=\I\tens\bh_{\XX}(y^*x_{(0)})x_{(1)}.
\end{split}
\]
Hence
\[
L_{(\id\tens\bh_{\GG})(\,\cdot\, y_{\eta})}^*L_{(\id\tens\bh_{\GG})(\,\cdot\,x_{\eta})}
=\I\tens\bh_{\XX}(y^*x_{(0)})L_{\bh_{\GG}(\,\cdot\,x_{(1)})}=\I\tens{L_{\bh_{\XX}(\,\cdot\,y)}^*}L_{\bh_{\XX}(\,\cdot\,x)}
\]
by the identity \eqref{EqInProd}. In particular, \eqref{EqNorm} holds.

We are now ready to obtain our contradiction. Consider the normal linear functionals $\eta_{ij}\in\Linf(\XX)_*$ forming the matrix coefficients of $\eta$,
\[
\eta(x)=\sum_{i,j=1}^n \eta_{ij}(x)e_{ij},
\]
and note that $\overline{\eta_{ij}}=\eta_{ji}$. Then let
\begin{equation}
L_{\eta}=\sum_{i,j=1}^n e_{ij}\tens{L_{\eta_{ij}}}\in{M_n(\CC)}\tens\c0(\widehat{\XX}),
\label{Leta}\end{equation}
which is an operator from $\CC^n\tens\Ltwo(\GG)$ to $\CC^n\tens\Ltwo(\XX)$. Using \eqref{EqInProd}, we find that for $x\in\Pol(\XX)$, $v\in\CC^n$ and $a\in\Pol(\GG)$ we have
\[
\begin{split}
L_{\eta}^*(\id\tens{L_{\bh_{\XX}(\,\cdot\,x)}})\bigl(v\tens\Lambda_{\GG}(a)\bigr)
&=\sum_{i,j=1}^n e_{ji}v\tens{L_{\eta_{ij}}^*}L_{\bh_{\XX}(\,\cdot\,x)}\Lambda_{\GG}(a)\\
&=\sum_{i,j=1}^n e_{ji}v\tens{L_{\eta_{ij}}^*}\bh_{\GG}\bigl(S(x_{(1)})a\bigr)\Lambda_{\XX}(x_{(0)})\\
&=\sum_{i,j=1}^n e_{ji}v\tens\bh_{\GG}\bigl(S(x_{(2)})a\bigr)\eta_{ji}(x_{(0)})\Lambda_{\GG}(x_{(1)})\\
&=\eta(x_{(0)})v\tens\bh_{\GG}\bigl(S(x_{(2)})a\bigr)\Lambda_{\GG}(x_{(1)})\\
&=\bigl(\eta(x_{(0)})\tens L_{\bh_{\GG}(\,\cdot\,x_{(1)})}\bigr)\bigl(v\tens\Lambda_{\GG}(a)\bigr)\\
&=L_{(\id\tens \bh_{\GG})(\,\cdot\,x_{\eta})}\bigl(v\tens\Lambda_{\GG}(a)\bigr).
\end{split}
\]

However, as $\dim{\Linf(\XX)}=\infty$, we can find by \eqref{EqSameComp} an infinite collection of mutually inequivalent irreducible representations $\{\iota\}$ of $\GG$ with corresponding central projections $p_\iota\in\c0(\hh{\GG})$ and $0\neq{x_\iota}\in\Pol(\XX)^{\iota}$ such that
\[
L_{\bh_{\XX}(\,\cdot\,x_\iota)}p_\iota=L_{\bh_{\XX}(\,\cdot\,x_\iota)}.
\]
Since
\[
\alpha(\Pol(\XX)^{\iota})\subset\Pol(\XX)^{\iota}\tens\Pol(\GG)^{\iota},
\]
we find that
\[
\begin{split}
\bigl\|L_{\bh_{\XX}(\,\cdot\,x_\iota)})\bigr\|
&=\bigl\|L_{(\id\tens\bh_{\GG})(\,\cdot\,(x_\iota)_{\eta})}\bigr\|\\
&=\bigl\|(\I\tens{p_\iota})L_{\eta}^*(\id\tens{L_{\bh_{\XX}(\,\cdot x_\iota)}})\bigr\|\\
&\leq\bigl\|L_{\eta}(\I\tens{p_\iota})\bigr\|\bigl\|L_{\bh_{\XX}(\,\cdot\,x_\iota)}\bigr\|.
\end{split}
\]
It follows that $\bigl\|L_{\eta}(\I\tens{p_\iota})\bigr\|\geq{1}$ for infinitely many $\iota$, which is a contradiction with \eqref{c_0Xformula} and \eqref{Leta}.
\end{proof}

\section{Actions of compact quantum groups on direct sums}\label{Adam}

In this section we consider an action of a compact quantum group on a direct sum of von Neumann algebras over a certain index set $\Ind$. We study a resulting relation on $\Ind$, classically corresponding to the relation of being in the same orbit of the action.

\subsection{Actions on direct sums of von Neumann algebras}
In the first subsection we consider the most general setup.
Let $\GG$ be a compact quantum group, and let $\sM$ be a von Neumann algebra of the following form:
\[
\sM=\prod_{i\in\Ind}\sM_i,
\]
where $\Ind$ is an index set and for each $i\in\Ind$ we have a von Neumann algebra $\sM_i$. We view each $\sM_i$ as a (non-unital) subalgebra of $\sM$ and write $p_i$ for the image of $\I_{\sM_i}$ in $\sM$ and $\bp_i$ for the canonical map $\sM\to\sM_i$. Elements of $\sM$ are norm bounded families $(x_i)_{i\in\Ind}$ with each $x_i\in\sM_i$. Given $x\in\sM$ we have $x=(x_i)_{i\in\Ind}$ where $x_i=\bp_i(x)$ for each $i\in\Ind$. Let $\alpha:\sM\to\sM\vtens\Linf(\GG)$ be an action of $\GG$ on $\sM$.

\begin{definition}\label{Def:Relation}
Let $i,j\in\Ind$. We say that $i$ is \emph{$\alpha$-related} to $j$ (which we write $i\Ralpha{j}$) if there exists $x\in{\sM_i}$ such that
\[
(\bp_j\tens\id)\alpha(x)\neq{0}.
\]
\end{definition}

Define for each $i,j\in\Ind$ the (usually non-unital) $*$-homomorphism $\alpha_{ij}:\sM_i\to\sM_j\vtens\Linf(\GG)$,
\[
\alpha_{ji}(x)=(\bp_j\tens\id)\alpha(x),\qqquad{x}\in\sM_i.
\]
Equation \eqref{actionEquation} now takes the form
\begin{equation}\label{action1}
\sum_{k\in\Ind}(\alpha_{jk}\tens\id)\bigl(\alpha_{ki}(x)\bigr)=(\id\tens\Delta_\GG)\bigl(\alpha_{ji}(x)\bigr),\qqquad{i},j\in\Ind,\;x\in\sM_i.
\end{equation}
Further say that $\alpha$ is \emph{implemented} by a unitary $U\in\B(\sH\tens\Ltwo(\GG))$, where $\sH$ is a Hilbert space, if there exists a (unital, normal) faithful representation $\pi:\sM\to\B(\sH)$ for which
\begin{equation}\label{implem}
(\pi\tens\id)\bigl(\alpha(y)\bigr)=U\bigl(\pi(y)\tens\I\bigr)U^*,\qqquad{y}\in\sM.
\end{equation}

The form of $\sM$ implies that in fact $\sH=\bigoplus\Limits_{i\in\Ind}\sH_i$ and $\pi(x)=\bigoplus\Limits_{i\in\Ind}\pi_i(x_i)$, where for each $i \in \Ind$ the map $\pi_i:\sM_i\to\B(\sH_i)$ is a unital representation of $\sM_i$. Thus we can in fact write the implementing unitary $U$ as a matrix $\bigl(U_{ij}\bigr)_{i,j\in\Ind}$, where $U_{ij}=\bigl(\pi(p_i)\tens\I\bigr)U\bigl(\pi(p_j)\tens\I\bigr)\in\B(\sH_j\tens\Ltwo(\GG),\sH_i\tens\Ltwo(\GG))$. In this picture equation \eqref{implem} can be written as
\begin{equation}\label{implem2}
(\pi_j\tens\id)\bigl(\alpha_{ji}(x)\bigr)=U_{ji}\bigl(\pi_i(x)\tens\I\bigr)U_{ji}^*,\qqquad{i},j\in\Ind,\;x\in\sM_i.
\end{equation}

An action $\alpha$ of $\GG$ on $\sM$ can always be implemented. One possible such implementation could be defined by choosing a faithful representation $\pi_0$ of $\sM$ on a Hilbert space $\sH_0$ and defining $\sH=\sH_0\tens\Ltwo(\GG)$, $\pi=(\pi_0\tens\id)\comp\alpha$ and $U=\ww^\GG_{23}\in\B(\sH_0)\vtens\linf(\hh{\GG})\vtens\Linf(\GG)\subset\B(\sH)\vtens\Linf(\GG)$. There are many other constructions yielding a unitary implementation (see e.g.~\cite{boca,Vaes-implementation}).

\begin{lemma}\label{lem:easy}
Let $i,j\in\Ind$. Then $i\Ralpha{j}$ if and only if $\alpha_{ji}\neq{0}$, if and only if $\alpha_{ji}(\I_{\sM_i})\neq{0}$. Moreover if $\alpha$ is implemented by a unitary $U$ as above the above conditions are equivalent to the fact that $U_{ji}\neq{0}$.
\end{lemma}

\begin{proof}
The first two equivalences are obvious. So is the third one, once we note that by \eqref{implem2} we have (for $i,j\in\Ind$)
\[
(\pi_j\tens\id)\bigl(\alpha_{ji}(\I_{\sM_i})\bigr)=U_{ji}U_{ji}^*.
\]
\end{proof}

\begin{proposition}\label{sym}
The relation $\Ralpha$ is symmetric.
\end{proposition}

\begin{proof}
We may and do assume that $\alpha$ is implemented by a unitary $U$ which is a representation of $\GG$. If $i\Ralpha{j}$, we have by Lemma \ref{lem:easy} that $U_{ji}\neq{0}$. Equivalently, there exist $\xi\in\sH_j\subset\sH$ and $\eta\in\sH_i\subset\sH$ such that
\[
(\omega_{\xi,\eta}\tens\id)(U)\neq{0}.
\]
Note however that $(\omega_{\xi,\eta}\tens\id)(U)$ belongs to $\Linf(\GG)$, and moreover, by \cite[Theorem 1.6]{mu}, it is in the domain of the antipode of $\GG$ and
\[
(\omega_{\xi,\eta}\tens\id)(U^*)=S\bigl((\omega_{\xi,\eta}\tens\id)(U)\bigr)\neq{0},
\]
where we use the fact that $S$ is injective. This means that $U_{ij}^*\neq{0}$, so $U_{ij}\neq{0}$ and by Lemma \ref{lem:easy} we see that $j\Ralpha{i}$.
\end{proof}

\subsection{Actions on direct sums of factors}

Easy classical examples (take $\GG=\mathbb{Z}_2$ acting on the set $X=\{1,2,3,4\}$ by flipping $1$ with $2$ and $3$ with $4$ and write $\sM=\Linf(X)$ as $\Linf(\{1\})\oplus\Linf(\{2,3\})\oplus\Linf(\{4\})$) show that in this generality our relation need not be transitive. The lack of transitivity however cannot happen as soon as the ``components'' $\sM_i$ are ``indecomposable''.

\begin{theorem}\label{thm:equiv}
Assume that each $\sM_i$ is a factor. Then the relation $\Ralpha$ is an equivalence relation. Furthermore, for each equivalence class $A$ of the relation $\Ralpha$ the action $\alpha$ restricts in an obvious way to an action on $\prod\Limits_{i\in{A}}\sM_i$. In particular, the projection $p_A=\sum\Limits_{j\in{A}}p_j$ is invariant under $\alpha$, i.e.~$\alpha(p_A)=p_A\tens\I$.
\end{theorem}

\begin{proof}
Factoriality and normality of the maps involved imply that for any $i,l\in\Ind$ the statement $i\Ralpha{l}$ is equivalent to $\alpha_{il}$ being injective. Thus if further $j\in\Ind$ and $l\Ralpha{j}$, then considering equation \eqref{action1} we see that the projection $(\id\tens\Delta_\GG)\bigl(\alpha_{ji}(\I_{\sM_i})\bigr)$ is equal to the sum of projections $\sum\Limits_{k\in\Ind}(\alpha_{jk}\tens\id)\bigl(\alpha_{ki}(\I_{\sM_i})\bigr)$, dominating the non-zero projection $(\alpha_{jl}\tens\id)\bigl(\alpha_{li}(\I_{\sM_i})\bigr)$ (as by the same token $\alpha_{jl}$ is injective). This shows that $\alpha_{ji}(\I_{\sM_i})\neq{0}$ and consequently that $\Ralpha$ is transitive.

Finally note that if $i\in\Ind$ then there must be $j\in\Ind$ such that $i\Ralpha{j}$ -- otherwise $\alpha$ would fail to be injective. This fact together with symmetry established in Proposition \ref{sym} and transitivity shown above implies that $i\Ralpha{i}$ and the proof of the first part is complete.

The second statement is clear.

The last statement follows from the fact that $p_A$ is the unit of the algebra $\prod\Limits_{i\in{A}}\sM_i$.
\end{proof}

We hence from now on assume that each $\sM_i$ is a factor.

\begin{lemma}\label{lem:ergodic}
Suppose that $\alpha$ is ergodic. Then $i\Ralpha{j}$ for any $i,j\in\Ind$.
\end{lemma}

\begin{proof}
We note that $p_A=\I$ for each equivalence class $A$ as in Theorem \ref{thm:equiv}.
\end{proof}

It is natural to expect that compact quantum groups cannot act on a direct sum of factors in such a way that the equivalence relation introduced above admits an infinite class. The result in full generality remains beyond our reach, but using Theorem \ref{mainthm} we can establish it for a direct sum of matrix algebras.

\begin{theorem}\label{conjfinite}
There is no action $\alpha$ of a compact quantum group $\GG$ on $\sM=\prod\Limits_{i\in\Ind}M_{n_i}(\CC)$ such that the relation $\Ralpha$ admits an infinite orbit.
\end{theorem}

\begin{proof}
By the last statement in Theorem \ref{thm:equiv} we can assume that the equivalence relation $\Ralpha$ has only one class.

Let $\alpha:\sM\to\sM\vtens\Linf(\GG)$ be an action as above. Consider a minimal non-zero projection $p\in\sM^\alpha$ (a minimal projection exists because $\sM^\alpha$ is a type $\mathrm{I}$ von Neumann algebra). Then $\alpha$ restricts to an ergodic action on $p\sM{p}$ (this follows from minimality of $p$). In particular $p\sM{p}$, which is again of the form $\prod\Limits_{j\in\Jnd}M_{k_j}(\CC)$, is finite-dimensional by Theorem \ref{mainthm}. In other words, $\Ind_p={\bigl\{i\in\Ind\st{p_ip}\neq{0}\bigr\}}$ is finite. Consider then $i\in\Ind_p$. Assume that there exists $j\in\Ind\setminus\Ind_p$. Then
\[
\alpha_{ji}(p_ip)=(\bp_j\tens\id)\bigl(\alpha(p_i p)\bigr)\leq(\bp_j\tens\id)\bigl(\alpha(p)\bigr)=\bp_j(p)\tens\I=0.
\]
As $i\Ralpha{j}$ by assumption, we obtain by simplicity of $M_{n_i}(\CC)$ and Lemma \ref{lem:easy} the contradiction $p_ip=0$. Hence $\Ind=\Ind_p$ is finite.
\end{proof}

Note that if a classical group acts on a von Neumann algebra $\sM$ then the action restricts to the center $\cZ(\sM)$. This is no longer true for quantum groups, and thus the last theorem cannot be reduced to the case when all summands of the direct sum are one-dimensional (but cf.~Theorem \ref{action}).

\begin{proposition}
There is no action of a compact quantum group of Kac type on a direct sum of a finite and infinite factor which has only one orbit.
\end{proposition}

\begin{proof}
Towards a contradiction suppose that $\sM$ is an infinite factor and $\sN$ a finite factor with trace $\tau$. Suppose that $\alpha:\sN\oplus\sM\to(\sN\oplus\sM)\tens\Linf(\GG)$ is an action with a unique orbit. Let $\eta:\sN\oplus\sM\to\sN$ be the canonical projection. The $*$-homomorphism $\jmath=(\eta\tens\id)\comp\bigl.\alpha\bigr|_{\sM}:\sM\to\sN\tens\Linf(\GG)$ is an embedding, but then $(\tau\tens\bh_\GG)\comp\jmath$ is a finite trace on $\sM$, which yields a contradiction.
\end{proof}

In particular a compact quantum group cannot act on $M_n(\CC)\oplus\B(\sH)$ for an infinite dimensional Hilbert space $\sH$ with only one orbit.

\subsection{Connectedness and torsion-freeness}

It turns out that having non-trivial orbits in an action on a direct sum of matrix algebras is related with notions of connectedness for compact quantum groups (\cite{Wangconnected}, see also \cite{PinzariConnected}) and torsion-freeness for discrete quantum groups (\cite{Meyer}). Our result can be stated in terms of the following definition:

\begin{definition}
A compact quantum group $\GG$ is said to satisfy the \emph{$(\mathrm{TO})$-condition} (``$\mathrm{TO}$'' standing for \emph{trivial orbits}) if for any action $\alpha$ of $\GG$ on $\sM=\prod\Limits_{i\in\Ind}M_{n_i}(\CC)$ the equivalence classes of $\Ralpha$ consist of single elements.
\end{definition}

It is easy to see that for a classical compact group $G$ the $(\mathrm{TO})$-condition is equivalent to connectedness (for the non-trivial direction see the theorem below). In the quantum setting the situation is more complicated.

\begin{theorem}
Let $\GG$ be a compact quantum group. If $\hh{\GG}$ is torsion free (in the sense of \cite{Meyer}) then $\GG$ satisfies the $(\mathrm{TO})$-condition, and if $\GG$ satisfies the $(\mathrm{TO})$-condition then $\GG$ is connected (in the sense of \cite{Wangconnected}). In general none of these implications can be reversed.
\end{theorem}

\begin{proof}
Assume first that $\hh{\GG}$ is torsion free and we have an action $\alpha$ of $\GG$ on $\sM=\prod\Limits_{i\in\Ind}M_{n_i}(\CC)$. By Theorem \ref{conjfinite} we may assume that $\Ind$ is finite. Then the definition of torsion-freeness implies however that the action $\alpha$ is (isomorphic to) a direct sum of actions on the components (see \cite[comments after Definition 3.1]{Voigt}) and the proof of the first implication is finished.

Suppose now that $\GG$ is not connected. This means that there is a non-trivial finite quantum group $\HH$ such that $\Pol(\HH)\subset\Pol(\GG)$ (as a Hopf $*$-subalgebra). Then the coproduct of $\HH$ can be interpreted as an action $\alpha$ of $\GG$ on $\Linf(\HH)=\Pol(\HH)$. However, the algebra $\Pol(\HH)$ is a direct sum of at least two matrix algebras by the existence of the counit, and it is easy to see (for example using the Haar state of $\HH$) that $\Ralpha$ must be a total relation. Thus $\GG$ does not satisfy the $(\mathrm{TO})$-condition.

To see that the converse to the second implication does not hold, it suffices to note that for $n\geq{4}$ the quantum permutation group $S_n^+$ is connected (\cite{Wangconnected}) and for its defining action $\alpha$ on $\bigoplus\Limits_{i=1}^n\CC$ the relation $\Ralpha$ has a single equivalence class. The first implication cannot be reversed, since for example the classical group $\mathrm{SO}(3)$ is connected and hence satisfies the $(\mathrm{TO})$-condition, but its dual is not torsion-free as it acts on $M_2(\CC)$ by the adjoint action with the non-trivial projective representation obtained by lifting along the two-fold covering $\mathrm{SU}(2)\twoheadrightarrow\mathrm{SO}(3)$ the fundamental representation of $SU(2)$. In fact, examples are precisely provided by those compact groups which are connected but whose fundamental group is not torsion-free, see the discussion in \cite[Section 7.2]{Meyer}.
\end{proof}

\subsection{Actions on countable discrete space}

We now specialize further to the situation where the discrete quantum space on which $\GG$ is acting is in fact classical and countable, i.e.~we consider an action $\alpha\colon\linf(\NN)\to\linf(\NN)\vtens\Linf(\GG)$ of a compact quantum group $\GG$ on $\NN$. In this context, for each $n\in\NN$ the projection $p_n$ is the characteristic function of $\{n\}$ and we let $\delta_n$ be the evaluation map $\linf(\NN)\ni{f}\mapsto{f(n)}\in\CC$. Then for $i,j\in\NN$ we define
\[
u_{i,j}=(\delta_i\tens\id)\alpha(p_j).
\]
It is clear that $\{u_{i,j}\}_{i,j\in\NN}$ are self-adjoint projections in $\Linf(\GG)$ and that for each $i\in\NN$
\begin{equation}\label{orth}
\sum_{j=1}^{\infty}u_{i,j}=\I
\end{equation}
by unitality of $\alpha$. Thus the $\{u_{i,j}\}_{j\in\NN}$ are pairwise orthogonal. We can write the value of $\alpha$ on $p_j$ as
\begin{equation}\label{alphapj}
\sum_{i=1}^{\infty}p_i\tens{u_{i,j}}
\end{equation}
with the sum strongly convergent. Clearly in this situation for any $i,j\in\NN$ we have $\alpha_{i,j}\neq{0}$ if and only if $u_{i,j}\neq{0}$, and so the relation $\Ralpha$ has a particularly simple description:
\[
\Bigl(\,k\Ralpha{l}\,\Bigr)\;\Longleftrightarrow\;\Bigl(\,u_{k,l}\neq{0}\,\Bigr).
\]
Using Theorem \ref{thm:equiv} we conclude that $\Ralpha$ is an equivalence relation. Indeed this follows from the identification $\linf(\NN)=\prod\Limits_{i\in\NN}\sM_i$ where $\sM_i=\CC$ for each $i$. The third statement of Theorem \ref{thm:equiv} and Theorem \ref{conjfinite} specialize in this context to the following two corollaries.

\begin{corollary}\label{cor:inv}
Let $i\in\NN$. Then $\sum\Limits_{j\Ralpha{i}}p_j$ is $\alpha$-invariant.
\end{corollary}

\begin{corollary}
The equivalence classes of $\Ralpha$ are of finite cardinality.
\end{corollary}

\begin{proposition}\label{countpreserve}
The counting measure is invariant for the action $\alpha$. In particular for any $j\in\NN$ we have
\begin{equation}\label{inv}
\sum_{i\Ralpha{j}}u_{i,j}=\I.
\end{equation}
\end{proposition}

\begin{proof}
We know by Corollary \ref{cor:inv} that all orbits of $\alpha$ are finite. By restriction, we obtain an action of $\GG$ on each orbit. By results of Wang \cite[Theorem 3.1 and Remark (2)]{qsym} the counting measure on each orbit is invariant for these restricted actions. It follows that the counting measure on $\NN$ is invariant for $\alpha$. Equation \eqref{inv} means precisely that the counting measure is preserved.
\end{proof}

Let us also make the following remarks:

\begin{remark}\label{hN}
Let $A$ be an equivalence class of $\Ralpha$. Then for any $i,j\in{A}$ we have \[\bh_\GG(u_{i,j})=\tfrac{1}{|A|}.\] Indeed, applying $(\id\tens\bh_\GG)$ to both sides of
\begin{equation}\label{Deluij}
\Delta_\GG(u_{i,j})=\sum\Limits_{k=1}^\infty{u_{i,k}}\tens{u_{k,j}},\qqquad{i,j}\in\NN,
\end{equation}
we obtain
\[
\bh_\GG(u_{i,j})\I=\sum_{k\in A}^\infty{u_{i,k}}\bh_\GG(u_{k,j}),
\]
where we can restrict the sum to the terms with $k\in{A}$, since only then $u_{i,k}$ is non-zero. Now the left hand side is non-zero since $u_{i,j}\neq 0$ and $\bh_\GG$ is faithful on $\Linf(\GG)$, so
\[
\I=\sum_{k\in{A}}{u_{i,k}}\tfrac{\bh_\GG(u_{k,j})}{\bh_\GG(u_{i,j})}.
\]
Now since $\{u_{i,k}\}_{k\in{A}}$ are pairwise orthogonal projections summing up to $\I$, the coefficients in the sum must all be equal to $1$. It follows that $\bh_{\GG}(u_{i,j})$ is independent of $i\in A$. Applying now $\bh_{\GG}$ to both sides of \eqref{inv}, we see that $\bh_{\GG}(u_{i,j})$ must be the constant $\tfrac{1}{|A|}$.
\end{remark}

\begin{remark}\label{ergodic=oneorbit}
In the case of an action $\alpha$ of $\GG$ on $\NN$, the converse of Lemma \ref{lem:ergodic} holds. Indeed, if $\alpha$ is not ergodic, let $p\in\linf(\NN)^{\alpha}$ be any non-zero projection different from the unit. Then $p$ is the characteristic function of some non-empty subset $I\subsetneqq\NN$, and clearly no point of $I$ can be equivalent to a point in its complement.
\end{remark}

As mentioned at the beginning of Section \ref{Adam}, the relation $\Ralpha$ is the quantum group analogue of the relation of being in the same orbit of the action of $\GG$. It can be shown that it formally coincides with the relation introduced in \cite[Section 4.2]{huang} albeit in the context of compact quantum group actions on \emph{compact} spaces in the \cst-algebraic framework.

Using Lemma \ref{lem:ergodic} and Theorem \ref{mainthm} we also see that there is no ergodic action $\alpha\colon\linf(\NN)\to\linf(\NN)\vtens\Linf(\GG)$. In what follows we shall give an elementary proof of this fact, independent of the more involved proof of Theorem \ref{mainthm} and using only Theorem \ref{thm:equiv} and Lemma \ref{lem:ergodic}.

\begin{theorem}\label{noact}
There is no ergodic action of a compact quantum group on $\NN$.
\end{theorem}

\begin{proof}
Let us assume that $\alpha\colon\linf(\NN)\to\linf(\NN)\vtens\Linf(\GG)$ is an ergodic action. In particular by Lemma \ref{lem:ergodic} we know that all $u_{i,j}$ are non-zero. Let $\bh_\NN$ be the invariant state on $\linf(\NN)$. For each $j$ we have
\begin{equation}\label{rhopj}
\bh_\NN(p_j)\I=\sum_{i=1}^\infty\bh_\NN(p_i)u_{i,j}.
\end{equation}
Moreover, since $\bh_\NN$ is a normal state
\begin{equation}\label{sumrho}
\sum_{i=1}^{\infty}\bh_\NN(p_i)=1.
\end{equation}

Now take a unit vector $\xi$ in the range of $u_{j,j}$. Applying both sides of \eqref{rhopj} to $\xi$ and taking scalar product with $\xi$ we obtain
\[
\bh_\NN(p_j)=\bh_\NN(p_j)+\sum_{i\neq{j}}\bh_\NN(p_i)\is{\xi}{u_{i,j}\xi}.
\]
Since all coefficients $\bh_\NN(p_i)$ in the above sum are non-zero, we find that $\is{\xi}{u_{i,j}\xi}=0$ for $i\neq{j}$, so the range of $u_{i,j}$ is orthogonal to that of $u_{j,j}$. Now fix $k\neq{j}$ and let $\eta$ be a unit vector in the range of $u_{k,j}$. Applying both sides of \eqref{rhopj} to $\eta$ and taking scalar product with $\eta$ we get
\[
\bh_\NN(p_j)=\sum_{i=1}^{\infty}\bh_\NN(p_i)\is{\eta}{u_{i,j}\eta}=\bh_\NN(p_j)\is{\eta}{u_{j,j}\eta}
+\bh_\NN(p_k)\is{\eta}{u_{k,j}\eta}+\sum_{k\neq{i}\neq{j}}^{\infty}\bh_\NN(p_i)\is{\eta}{u_{i,j}\eta}.
\]
However, the first term on the right hand side is zero because the range of $u_{j,j}$ is orthogonal to the range of $u_{k,j}$. This means that
\[
\bh_\NN(p_j)=\bh_\NN(p_k)\is{\eta}{u_{k,j}\eta}
+\sum_{k\neq{i}\neq{j}}^{\infty}\bh_\NN(p_i)\is{\eta}{u_{i,j}\eta}
\]
so $\bh_\NN(p_j)\leq\bh_\NN(p_k)$. Exchanging the roles of $j$ and $k$ yields opposite inequality and we find
\[
\bh_\NN(p_j)=\bh_\NN(p_k)
\]
for all $k,j\in\NN$ which contradicts \eqref{sumrho}.
\end{proof}

\section{Applications}

\subsection{Quantum Clifford theory}\label{QCliffordTheory}

Before stating our generalization of the theorem of Clifford (\cite[Theorem 1]{clifford}) let us note one result concerning normal subgroups of locally compact quantum groups. Recall that a closed quantum subgroup $\HH$ of a locally compact quantum group $\GG$ is normal if and only if $\Linf(\hh{\HH})$ is invariant under the natural action of $\GG$ on $\Linf(\hh{\GG})$ by ``inner automorphisms'' (\cite{extensions}, \cite[Section 4]{centers}). It turns out that in this case the action actually restricts to an action on the center of $\Linf(\hh{\HH})$ which translates into an action on a classical space.

\begin{theorem}\label{action}
Let $\GG$ be a locally compact quantum group and $\HH\subset\GG$ a normal closed quantum subgroup of $\GG$. Then for any $x\in\cZ\bigl(\Linf(\hh{\HH})\bigr)$ we have
\[
\ww^\GG(x\tens\I){\ww^\GG}^*\in\cZ\bigl(\Linf(\hh{\HH})\bigr)\vtens\Linf(\GG).
\]
In particular $x\mapsto\ww^\GG(x\tens\I){\ww^\GG}^*$ is an action of $\GG$ on $\cZ\bigl(\Linf(\hh{\HH})\bigr)$.
\end{theorem}

\begin{proof}
Normality of $\HH$ means that
\[
\ww^\GG\bigl(\Linf(\hh{\HH})\tens\I\bigr){\ww^\GG}^*\subset
\Linf(\hh{\HH})\vtens\Linf(\GG).
\]
Applying the tensor product of the respective unitary antipodes $R^{\hh{\GG}}\tens{R^\GG}$ to both sides we get
\[
{\ww^\GG}^*\bigl(\Linf(\hh{\HH})\tens\I\bigr)\ww^\GG\subset\Linf(\hh{\HH})\vtens\Linf(\GG).
\]
Take $y\in\Linf(\hh{\HH})$ and $x\in\cZ\bigl(\Linf(\hh{\HH})\bigr)$. We have
\[
\begin{split}
\ww^\GG(x\tens\I){\ww^\GG}^*(y\tens\I)&=\ww^\GG(x\tens\I){\ww^\GG}^*(y\tens\I)\ww^\GG{\ww^\GG}^*\\
&=\ww^\GG{\ww^\GG}^*(y\tens\I)\ww^\GG(x\tens\I){\ww^\GG}^*\\
&=(y\tens\I)\ww^\GG(x\tens\I){\ww^\GG}^*
\end{split}
\]
which means that the left leg of $\ww^\GG(x\tens\I){\ww^\GG}^*$ belongs to $\cZ\bigl(\Linf(\hh{\HH})\bigr)$.

The fact that $x\mapsto\ww^\GG(x\tens\I){\ww^\GG}^*$ is an action of $\GG$ on $\cZ\bigl(\Linf(\hh{\HH})\bigr)$ is easily checked (cf.~\cite[Section 4]{centers}).
\end{proof}

Note further that if the subgroup $\HH$ is compact then $\linf(\hh{\HH})$ is a direct sum of matrix algebras and $\cZ\bigl(\linf(\hh{\HH})\bigr)$ is naturally isomorphic to the algebra of bounded functions on the set $\Irr{\HH}$ of equivalence classes of irreducible unitary representations of $\HH$. This, in particular, happens if $\GG$ is compact (cf.~\cite[Section 6]{DKSS}), so in this case we obtain an action of $\GG$ on $\Irr{\HH}$; we will return to investigating it below, but as suggested in the introduction we begin from a more general context of arbitrary quantum subgroups of discrete quantum groups.

Let $\GGamma$ be a discrete quantum group and let $\LLambda$ be a quantum subgroup of $\GGamma$. Then $\LLambda$ is automatically discrete and open in $\GGamma$ (\cite{KalKS}). In particular we have a surjective normal map $\bpi\colon\linf(\GGamma)\to\linf(\LLambda)$ commuting with comultiplications. The quantum homogeneous space $\LLambda\backslash\GGamma$ is defined by setting
\[
\linf(\LLambda\backslash\GGamma)=\bigl\{x\in\linf(\GGamma)\st(\bpi\tens\id)\Delta_\GGamma(x)=\I\tens{x}\bigr\}.
\]
Moreover the right action of $\hh{\GGamma}$ on $\GGamma$ (\cite[Section 4]{centers}) restricts to an action of $\hh{\GGamma}$ on $\LLambda\backslash\GGamma$, i.e.~we have
\[
\ww^{\hh{\GGamma}}\bigl(\linf(\LLambda\backslash\GGamma)\tens\I\bigr){\ww^{\hh{\GGamma}}}^*
\subset\linf(\LLambda\backslash\GGamma)\vtens\Linf(\hh{\GGamma})
\]
(\cite[Propositions 4.3 \& 3.4]{ext}) which defines an action $\alpha:\linf(\LLambda\backslash\GGamma)\to\linf(\LLambda\backslash\GGamma)\vtens\Linf(\hh{\GGamma})$. Note that the fixed point algebra of this action equals
$\cZ\bigl(\linf(\GGamma)\bigr) \cap \linf(\LLambda\backslash\GGamma) $.

As $\linf(\LLambda\backslash\GGamma)$ is a subalgebra of
\[
\linf(\GGamma)=\prod_{\iota\in\Irr{\hh{\GGamma}}}M_{n_\iota}(\CC)
\]
it is itself isomorphic to a product of matrix algebras:
\[
\linf(\LLambda\backslash\GGamma)=\prod_{i\in\Ind}\sM_i
\]
with each $\sM_i$ isomorphic to $M_{m_i}(\CC)$ for some $m_i\in\NN$ and we are in the situation described in Section \ref{Adam}. In particular we will consider the equivalence relation $\Ralpha$ on $\Ind$ associated to the action $\alpha$. Departing from our earlier convention, for each $i\in\Ind$ we will write $\I_i$ for the unit of $\sM_i$ considered as an element of $\linf(\GGamma)$. Note that although the unit of $\sM_i$ is a minimal central projection in $\linf(\LLambda\backslash\GGamma)$, it is usually not central (nor minimal) in $\linf(\GGamma)$.

\begin{theorem}\label{thm:act_sub}
Let $\GGamma$, $\LLambda$ and $\alpha$ be as above. Then for all $i\in\Ind$ the element
\[
\sum\Limits_{j\Ralpha{i}}\I_j\in\linf(\LLambda\backslash\GGamma)\subset\linf(\GGamma)
\]
is the central support $\z(\I_i)$ in $\linf(\GGamma)$ of the projection $\I_i$. Moreover $\z(\I_i)$ is orthogonal to $\z(\I_j)$ if $i$ is not equivalent to $j$.
\end{theorem}

In particular for any $\kappa\in\Irr{\hh{\GGamma}}$ there exists $i\in\Ind$ such that
\begin{enumerate}
\item\label{thm:act_sub1} for all $j\in\Ind$ we have $p_\kappa\I_j\neq{0}$ if and only if $j\Ralpha{i}$,
\item\label{thm:act_sub2} we have $p_\kappa\biggl(\,\sum\Limits_{j\Ralpha{i}}\I_j\biggr)=p_\kappa$.
\end{enumerate}

\begin{proof}[Proof of theorem \ref{thm:act_sub}]
For $i\in\Ind$ denote by $\supp_\GGamma{\I_i}$ the set $\bigl\{\iota\in\Irr{\hh{\GGamma}}{\st{p_\iota}\I_i\neq{0}\bigr\}}$, which is the support of the central projection of $\I_i$ in $\linf(\GGamma)$.

Now take $\kappa\in\Irr{\hh{\GGamma}}$ and $i\in\Ind$. Suppose $\kappa\not\in\supp_\GGamma(\I_i)$, so that $p_\kappa\I_i=0$. We have
\begin{equation}\label{eq:av11}
\begin{split}
0=\ww^{\hh{\GGamma}}(p_\kappa\I_i\tens\I){\ww^{\hh{\GGamma}}}^*&=(p_\kappa\tens\I)\alpha(\I_i)\\
&=(p_\kappa\tens\I)\sum_{j\Ralpha{i}}\alpha_{ji}(\I_i)=\sum_{j\Ralpha{i}}(p_\kappa\tens\I)\alpha_{ji}(\I_i).
\end{split}
\end{equation}
Note that for $j\Ralpha{i}$ the positive element $\alpha_{ji}(\I_i)$ is non-zero.

Applying $\id\tens\bh_{\hh{\GGamma}}$ to both sides of \eqref{eq:av11} we obtain
\[
0=\sum_{j\Ralpha{i}}p_\kappa(\id\tens\bh_{\hh{\GGamma}})\alpha_{ji}(\I_i).
\]
As all the terms are positive, this implies
\[
p_\kappa(\id\tens\bh_{\hh{\GGamma}})\alpha_{ji}(\I_i)=0,\qqquad{j}\Ralpha{i},
\]
Denoting $(\id\tens\bh_{\hh{\GGamma}})\alpha_{ji}(\I_i)$ by $x_j$ we find (due to centrality of $p_\kappa$) that
\[
p_\kappa\bigl(\linf(\LLambda\backslash\GGamma)x_j\linf(\LLambda\backslash\GGamma)\bigr)=0,
\]
but $\linf(\LLambda\backslash\GGamma)x_j\linf(\LLambda\backslash\GGamma)$ is all of $\sM_j$ because $x_j\neq{0}$ (by faithfulness of $\bh_{\hh{\GGamma}}$) and $\sM_j$ is a matrix algebra.

This means that for any $j\Ralpha{i}$ we have $\kappa\not\in\supp_\GGamma{\I_j}$. By symmetry, we find that
\[
\supp_\GGamma\I_j=\supp_\GGamma\I_i,\qqquad{j\Ralpha{i}}.
\]

Now by the last part of Theorem \ref{thm:equiv} the element $\sum\Limits_{j\Ralpha{i}}\I_j$ is invariant under the action $\alpha$, i.e.
\[
\sum\Limits_{j\Ralpha{i}}\I_j\in\cZ\bigl(\linf(\GGamma)\bigr).
\]
In particular, denoting by $\z(\cdot)$ the central support of a projection in the algebra $\linf(\GGamma)$ we find that $\sum\Limits_{j\Ralpha{i}}\I_j\geq\z(\I_i)$.

Conversely, let $r$ be a projection in $\cZ\bigl(\linf(\GGamma)\bigr)$ and $r\geq\I_i$. Since for any $j\Ralpha{i}$ we have $\supp_\GGamma{\I_j}=\supp_\GGamma{\I_i}$, we see that if $r\geq\I_i$ then $r\geq\I_j$ for all $j\Ralpha{i}$ and consequently $r\geq\sum\Limits_{j\Ralpha{i}}\I_j$. It follows that $\z(\I_i)=\sum\Limits_{j\Ralpha{i}}\I_j$.

To finish the proof we have to show that $\z(\I_i)$ and $\z(\I_j)$ are orthogonal when $i$ and $j$ are not equivalent. However, we have then that
\[
0=\underset{j'\Ralpha{j}}{\underset{i'\Ralpha{i}}{\sum}}(\I_{j'}\tens\I)\alpha(\I_{i'})
=\bigl(\z(\I_j)\tens\I\bigr)\alpha\bigl(\z(\I_i)\bigr)=\z(\I_j)\z(\I_i)\tens\I,
\]
from which the statement follows.
\end{proof}

\begin{remark}\label{supports}
The proof of Theorem \ref{thm:act_sub} shows in particular that $i\Ralpha{j}$ if and only if $\supp_\GGamma\I_i=\supp_\GGamma\I_j$.
\end{remark}

The classical theorem of Clifford (\cite[Theorem 1]{clifford}) describes the restriction of an irreducible representation of a group $G$ to a normal subgroup $H$. Theorem \ref{thm:act_sub} provides a generalization of this result for compact quantum groups: let $\GG$ be a compact quantum group and $\HH$ its normal closed quantum subgroup. Representations of $\GG$ and $\HH$ correspond bijectively (and functorially) to normal representations of the von Neumann algebras $\linf(\hh{\GG})$ and $\linf(\hh{\HH})$ respectively. Functoriality of the correspondence means that all the structure of representations of $\GG$ and $\HH$ is preserved. This in particular applies to direct sums, tensor products, decomposition into irreducible representations etc. Each irreducible representation $\rho$ of $\linf(\hh{\HH})$ corresponds to a unique central projection $\I_\rho$ (its central cover, see \cite[Section 3.8.1]{pedersen}) and similarly for representations of $\linf(\hh{\GG})$. The algebra $\linf(\hh{\HH})$ is embedded into $\linf(\hh{\GG})$ so it makes sense to apply representations of $\linf(\hh{\GG})$ to elements of $\linf(\hh{\HH})$. Taking $\GGamma=\hh{\GG}$ and $\LLambda=\hh{\GG/\HH}$ we obtain a pair $(\GGamma,\LLambda)$ to which Theorem \ref{thm:act_sub} applies. In this case $\LLambda$ is also normal, so $\linf(\LLambda\backslash\GGamma)=\linf(\GGamma/\LLambda)=\linf(\hh{\HH})$ and the action $\alpha$ discussed above is an action of $\GG$ on $\linf(\hh{\HH})$. Moreover, by Theorem \ref{action} this action may be restricted to $\cZ\bigl(\linf(\hh{\HH})\bigr)$ and thus yields an action of $\GG$ on the classical set $\Irr{\HH}$. Clearly the equivalence relation on $\Irr{\HH}$ obtained from the action of $\GG$ on $\linf(\hh{\HH})=\prod\Limits_{\sigma\in\Irr{\HH}}M_{n_\sigma}(\CC)$ is the same as the one coming from the restriction of this action to $\cZ\bigl(\linf(\hh{\HH})\bigr)\cong\linf(\Irr{\HH})$ (cf.~Lemma \ref{lem:easy}). In this situation Theorem \ref{thm:act_sub} takes the following form:

\begin{theorem}\label{qCliff}
Let $\GG$ be a compact quantum group and $\HH$ a closed normal quantum subgroup of $\GG$. Then for an irreducible representation $\sigma$ of $\linf(\hh{\HH})$ the element
\[
\sum\Limits_{\rho\Ralpha\sigma}\I_\rho\in\linf(\hh{\HH})\subset\linf(\hh{\GG})
\]
is the central support $\z(\I_{\sigma})$ in $\linf(\hh{\GG})$ of the projection $\I_\sigma$, and $\z(\I_{\sigma})$ is orthogonal to $\z(\I_{\rho})$ for $\sigma$ and $\rho$ not in the same orbit.
\end{theorem}

The interpretation of Theorem \ref{qCliff} is that an irreducible representation $\pi$ of the compact quantum group $\GG$ when restricted to the normal quantum subgroup $\HH$ is equivalent to a direct sum of irreducible representations of $\HH$ which form precisely one orbit under the action of $\GG$. Note also that the equivalence class of the trivial representation of $\HH$ with respect to $\Ralpha$ consists of one element. Thus the above theorem says in particular that if a restriction of an irreducible representation of a compact quantum group $\GG$ to a normal quantum subgroup $\HH$ contains the trivial representation of $\HH$, then in fact this restriction is a multiple of the trivial representation. This is in fact the original definition of normality for subgroups of compact quantum groups in \cite[Page 679]{WangFree} (cf.~\cite[Proposition 2.1]{Wangconnected}). Furthermore, if $\HH$ is a \emph{central} subgroup of $\GG$, as defined in \cite{Wangconnected} (see also \cite{Patri}), the algebra $\linf(\hh{\HH})$ is contained in $\cZ\bigl(\linf(\hh{\GG})\bigr)$ by \cite{centers}, so that the action $\alpha$ discussed above trivializes, each orbit consists of one element (which must be an irreducible representation of $\HH$ of dimension 1, as $\HH$ is cocommutative) and Theorem \ref{qCliff} gives another proof of the forward implication of \cite[Theorem 6.3]{Patri}.

In classical Clifford theory, we consider the action of a group $G$ on the set or irreducible representations of its normal subgroup $H$ by composing with inner automorphisms. It is clear that this action preserves dimension, i.e.~irreducible representations of $H$ belonging to one class are all of the same dimension. In the setting of quantum groups we have the following result.

\begin{proposition}\label{dimprop}
Let $\GG$ be a compact quantum group of Kac type and let $\HH$ be a closed normal quantum subgroup of $\GG$. Then any two irreducible representations $\sigma$ and $\tau$ of $\HH$ in the same orbit have the same dimension. Moreover, if $\pi$ is any irreducible representation of $\GG$ with $\pi(\I_{\sigma})\neq 0$, then also the multiplicity of $\sigma$ in $\pi$ is the same as the multiplicity of $\tau$ in $\pi$.
\end{proposition}

\begin{proof}
Fix $\sigma$ an irreducible representation of $\HH$, and $\pi$ an irreducible representation of $\GG$ such that $\pi(\I_{\sigma})\neq 0$. Consider the action of $\GG$ on $M_{n_\pi}(\CC)$ given by the formula
\[
x\longmapsto{V}(x\tens\I)V^*,
\]
where $V=(\pi\tens\id)\ww^\GG$. As the fixed point algebra must in this case be contained in the center of $M_{n_\pi}(\CC)$, the action is ergodic. Let $\Tr$ denote the standard (non-normalized) trace on the summand $M_{n_\pi}(\CC)$ of $\linf(\GGamma)$. As $\GG$ is of Kac type, the invariant state must be tracial, and hence $\Tr$ is preserved. By Proposition \ref{countpreserve} its restriction to $\pi\bigl(\cZ\bigl(\linf(\hh{\HH})\bigr)\bigr)$ must be a multiple of the counting measure, hence
\[
\dim\pi(\I_{\sigma})=\Tr\bigl(\pi(\I_{\sigma})\bigr)=\Tr\bigl(\pi(\I_{\tau})\bigr)=\dim\pi(\I_{\tau})
\]
for $\sigma$ and $\tau$ in the same orbit.

On the other hand, let
\[
B=\bigoplus_{\tau\Ralpha\sigma}M_{n_{\tau}}(\CC)\subset\linf(\hh{\HH}),
\]
and write $e_{kl}^{\tau}$ for the associated matrix units. Once again considering an appropriate restriction of the action $\alpha$ to $B$ we first deduce that the fixed point algebra must be contained in the center of $B$, and then arguing as in Remark \ref{ergodic=oneorbit} we see that this restriction is ergodic. As $\GG$ is of Kac type, it follows from \cite[Proposition 20]{DFY} (see also \cite{BanicaSym}) that the Markov trace
\[
\Tr_{\text{\rm{\tiny{M}}}}(e_{kl}^{\tau})=\delta_{kl}n_{\tau}
\]
is up to a scalar the unique invariant functional on $B$. As also $\bigl.\Tr\comp\pi\bigr|_B$ is a non-zero invariant positive functional on $B$, we deduce that there exists a positive scalar $c_{\pi}>0$ such that for any $k$
\[
c_{\pi}n_{\sigma}= c_{\pi}\Tr_{\text{\rm{\tiny{M}}}}(e_{kk}^{\sigma})=\Tr\bigl(\pi(e_{kk}^{\sigma})\bigr)=\mult_{\pi}(\sigma),
\]
with $\mult_{\pi}(\sigma)$ the multiplicity of $\sigma$ inside $\pi$.

It thus follows that
\[
\dim\pi(\I_{\sigma})=\dim(\sigma)\mult_{\pi}(\sigma)=c_{\pi}n_{\sigma}^2,
\]
and hence $\sigma\mapsto{n_{\sigma}}$ and $\sigma\mapsto\mult_{\pi}(\sigma)$ are constant on orbits.
\end{proof}

\subsection{Vergnioux-Voigt equivalence}\label{VVeq}

So far, given a discrete quantum group $\GGamma$ with a quantum subgroup $\LLambda$, we defined an action of $\hh{\GGamma}$ on $\linf(\LLambda\backslash\GGamma)$ and hence an equivalence relation on the index set $\Ind$ of the decomposition
\[
\linf(\LLambda\backslash\GGamma)=\prod_{i\in\Ind}\sM_i
\]
with each $\sM_i$ a matrix algebra. On the other hand in \cite{orientation} an equivalence relation on $\Irr{\hh{\GGamma}}$ was introduced as follows: since $\Linf(\hh{\LLambda})$ is an invariant subalgebra of $\Linf(\hh{\GGamma})$, any representation of $\hh{\LLambda}$ is, in particular, a representation of $\hh{\GGamma}$, so we may see $\Irr{\hh{\LLambda}}\subset\Irr{\hh{\GGamma}}$ and we can take tensor products between elements of $\Irr{\hh{\GGamma}}$ and $\Irr{\hh{\LLambda}}$. Take now $\sigma,\tau\in\Irr{\hh{\GGamma}}$. We write $\sigma\RLambda\tau$ if there exists $\gamma\in\Irr{\hh{\LLambda}}$ such that $\tau\subset\sigma\tp\gamma$. This equivalence relation was later put to use e.g.~in \cite{VergniouxVoigt}.

It turns out that the equivalence relations $\RLambda$ and $\Ralpha$ are related (note that these are equivalence relations on different sets). In order to state this relationship recall from the proof of Theorem \ref{thm:act_sub} that for $i\in\Ind$ -- the set indexing simple summands of $\linf(\LLambda\backslash\GGamma)$ -- the symbol $\supp_\GGamma{\I_i}$ denotes the set of those $\iota\in\Irr{\hh{\GGamma}}$ for which $p_\iota\I_i\neq{0}$. Recall form Remark \ref{supports} that $i\Ralpha{j}$ if and only if $\supp_\GGamma\I_i=\supp_\GGamma\I_j$.

\begin{theorem}
Let $\GGamma$ be a discrete quantum group and let $\LLambda$ be its quantum subgroup. Then two elements $\sigma,\tau\in\Irr{\hh{\GGamma}}$ satisfy $\sigma\RLambda\tau$ if and only if there exists $i\in\Ind$ such that $\sigma,\tau\in\supp_\GGamma(\I_i)$.
\end{theorem}

\begin{proof}
We will identify a representation $\pi$ of $\hh{\GGamma}$ with the corresponding representation of the von Neumann algebra $\linf(\GGamma)$ and, as before, $p_\pi$ will denote the corresponding central projection in $\linf(\hh{\GGamma})$.

As we already noted, $\LLambda$ is an open quantum subgroup of $\GGamma$. Let $\I_\LLambda$ be the support of $\LLambda$ (\cite[Section 2]{KalKS}). Then $\I_\LLambda\in\linf(\LLambda\backslash\GGamma)$ and thus $\Delta_\GGamma(\I_\LLambda)\in\linf(\LLambda\backslash\GGamma)\vtens\linf(\GGamma)$. By \cite[Theorem 3.3]{KalKS} for each $i\in\Ind$ we have $\Delta_\GGamma(\I_\LLambda)(\I_i\tens\I)\neq{0}$ and moreover each of these elements is positive.

By \cite[Equation (3.6) \& following remarks]{KalKS} we have
\[
\Delta_\GGamma(\I_\LLambda)(\I\tens{x})=\Delta_\GGamma(\I_\LLambda)\bigl(R^\GGamma(x)\tens\I\bigr)
\]
for all $x\in\linf(\GGamma/\LLambda)=R^\GGamma\bigl(\linf(\LLambda\backslash\GGamma)\bigr)$. Thus
\[
\Delta_\GGamma(\I_\LLambda)\bigl(\I\tens{R^\GGamma(y)}\bigr)=\Delta_\GGamma(\I_\LLambda)(y\tens\I),\qqquad{y}\in\linf(\LLambda\backslash\GGamma).
\]
Using this with $y=\I_j$ for $j\in\Ind$ we obtain
\[
\Delta_\GGamma(\I_\LLambda)(\I_j\tens\I)
=\Delta_\GGamma(\I_\LLambda)(\I_j\tens\I)^2=\Delta_\GGamma(\I_\LLambda)\bigl(\I\tens{R^\GGamma(\I_j)}\bigr)(\I_j\tens\I)
=\Delta_\GGamma(\I_\LLambda)\bigl(\I_j\tens{R^\GGamma(\I_j)}\bigr).
\]
Now summing over $j$ we obtain
\begin{equation}\label{DelI}
\Delta_\GGamma(\I_\LLambda)=\sum_{j\in\Ind}\Delta_\GGamma(\I_\LLambda)\bigl(\I_j\tens{R^\GGamma(\I_j)}\bigr).
\end{equation}

Now let $\sigma,\tau\in\Irr{\hh{\GGamma}}$ be such that $\sigma,\tau\in\supp_\GGamma\I_i$ for some $i\in\Ind$. This means precisely that
\[
\sigma(\I_i)\tens\tau^\cc\bigl(R^\GGamma(\I_i)\bigr)\neq{0}.
\]
which in turn is equivalent to
\[
\ker{\bigl.\sigma\tens\tau^\cc\bigr|_{\sM_i\tens{R^\GGamma(\sM_i)}}}=\{0\}
\]
by simplicity of $\sM_i\tens{R^\GGamma(\sM_i)}$. Now let us apply $\sigma\tens\tau^\cc$ to both sides of \eqref{DelI}:
\[
(\sigma\tens\tau^\cc)\Delta_\GGamma(\I_\LLambda)=\sum_{j\in\Ind}(\sigma\tens\tau^\cc)\bigl(\Delta_\GGamma(\I_\LLambda)(\I_j\tens{R^\GGamma(\I_j)})\bigr).
\]
We see that at least one of the summands on the right hand side is non zero. But as all terms are positive, we find that
\[
(\sigma\tens\tau^\cc)\Delta_\GGamma(\I_\LLambda)\neq{0}
\]
or, in other words, $(\sigma\tp\tau^\cc)(\I_\LLambda)\neq{0}$. Now from the fact that $\I_\LLambda=\sum\Limits_{\gamma\in\Irr{\hh{\LLambda}}}p_\gamma$ (\cite[Section 2.4]{fima}) we immediately obtain that there exists $\gamma\in\Irr{\hh{\LLambda}}$ such that $\gamma\subset\sigma\tp\tau^\cc$ which is equivalent to $\tau\subset\sigma\tp\gamma$ by \cite[Proposition 3.2]{q-lorentz}. Conversely, if $\sigma\RLambda\tau$ the above argument can be backtracked to prove the existence of an element $i$ with $\sigma,\tau\in\supp_{\GGamma}(\I_i)$, proving the theorem.
\end{proof}

\subsection*{Acknowledgements}

The second and fourth authors were partially supported by the National Science Center (NCN) grant no.~2015/17/B/ST1/00085. The third author was partially supported by the National Science Center (NCN) grant no.~2014/14/E/ST1/00525. The first author was partially supported by the FWO grant G.0251.15N.~and by the grant H2020-MSCA-RISE-2015-691246-QUANTUM DYNAMICS.


\printbibliography

\end{document}